\documentclass[oneside,american,english]{amsart}
\usepackage[T1]{fontenc}
\usepackage[latin9]{inputenc}
\usepackage{amsthm}
\usepackage{amstext}
\usepackage{amssymb}
\usepackage{esint}

\makeatletter
\numberwithin{equation}{section}
\numberwithin{figure}{section}
  \theoremstyle{plain}
  \newtheorem*{thm*}{\protect\theoremname}
\theoremstyle{plain}
\newtheorem{thm}{\protect\theoremname}
  \theoremstyle{definition}
  \newtheorem{defn}[thm]{\protect\definitionname}
  \theoremstyle{definition}
  \newtheorem{example}[thm]{\protect\examplename}
  \theoremstyle{plain}
  \newtheorem{lem}[thm]{\protect\lemmaname}
  \newcounter{casectr}
  \newenvironment{caseenv}
  {\begin{list}{{\itshape\ \protect\casename} \arabic{casectr}.}{%
   \setlength{\leftmargin}{\labelwidth}
   \addtolength{\leftmargin}{\parskip}
   \setlength{\itemindent}{\listparindent}
   \setlength{\itemsep}{\medskipamount}
   \setlength{\topsep}{\itemsep}}
   \setcounter{casectr}{0}
   \usecounter{casectr}}
  {\end{list}}
  \theoremstyle{remark}
  \newtheorem{rem}[thm]{\protect\remarkname}
  \theoremstyle{plain}
  \newtheorem{prop}[thm]{\protect\propositionname}
  \theoremstyle{plain}
  \newtheorem{cor}[thm]{\protect\corollaryname}

\usepackage{ae,aecompl} 

\makeatother

\usepackage{babel}
  \addto\captionsamerican{\renewcommand{\casename}{Case}}
  \addto\captionsamerican{\renewcommand{\corollaryname}{Corollary}}
  \addto\captionsamerican{\renewcommand{\definitionname}{Definition}}
  \addto\captionsamerican{\renewcommand{\examplename}{Example}}
  \addto\captionsamerican{\renewcommand{\lemmaname}{Lemma}}
  \addto\captionsamerican{\renewcommand{\propositionname}{Proposition}}
  \addto\captionsamerican{\renewcommand{\remarkname}{Remark}}
  \addto\captionsamerican{\renewcommand{\theoremname}{Theorem}}
  \addto\captionsenglish{\renewcommand{\casename}{Case}}
  \addto\captionsenglish{\renewcommand{\corollaryname}{Corollary}}
  \addto\captionsenglish{\renewcommand{\definitionname}{Definition}}
  \addto\captionsenglish{\renewcommand{\examplename}{Example}}
  \addto\captionsenglish{\renewcommand{\lemmaname}{Lemma}}
  \addto\captionsenglish{\renewcommand{\propositionname}{Proposition}}
  \addto\captionsenglish{\renewcommand{\remarkname}{Remark}}
  \addto\captionsenglish{\renewcommand{\theoremname}{Theorem}}
  \providecommand{\casename}{Case}
  \providecommand{\corollaryname}{Corollary}
  \providecommand{\definitionname}{Definition}
  \providecommand{\examplename}{Example}
  \providecommand{\lemmaname}{Lemma}
  \providecommand{\propositionname}{Proposition}
  \providecommand{\remarkname}{Remark}
  \providecommand{\theoremname}{Theorem}
\providecommand{\theoremname}{Theorem}

\begin{document}
\selectlanguage{american}%
\global\long\def\o{\mathbf{1}}

\global\long\def\epsilon{\varepsilon}

\global\long\def\phi{\varphi}

\global\long\def\L{\mathcal{L}}

\global\long\def\Sa{\mathcal{S}_{\alpha}}

\global\long\def\R{\mathbb{R}}

\global\long\def\N{\mathbb{N}}

\global\long\def\T{\mathcal{T}}

\global\long\def\S{\mathcal{S}}

\global\long\def\Kn{\mathcal{K}^{n}}

\global\long\def\Hn{\mathcal{H}^{n}}

\global\long\def\lc{\textrm{LC}\!\left(\R^{n}\right)}

\global\long\def\cvx{\textrm{Cvx}\!\left(\R^{n}\right)}

\global\long\def\cvxo{\textrm{Cvx}_{0}}

\global\long\def\cao{C_{\alpha}^{0}}

\global\long\def\Ca{C_{\alpha}\left(\R^{n}\right)}

\global\long\def\base{\text{base}_{\alpha}}

\global\long\def\Im{\text{Im}}

\selectlanguage{english}%

\title{Support functions and mean width for $\alpha$-concave functions}
\begin{abstract}
In this paper we extend some notions, previously defined for log-concave
functions, to the larger domain of so-called $\alpha$-concave functions.
We begin with a detailed discussion of support functions -- first
for log-concave functions, and then for general $\alpha$-concave
functions. We continue by defining mean width, and proving some basic
results such as an Urysohn type inequality. Finally, we demonstrate
how such geometric results can imply Poincaré type inequalities.
\end{abstract}


\author{Liran Rotem}

\address{School of Mathematical Sciences, Tel Aviv University, Ramat Aviv,
69978, Tel Aviv, Israel}

\email{liranro1@post.tau.ac.il}

\maketitle

\section{\label{sec:support-functions}Support functions and $\alpha$-concave
functions}

A well known construction in classic convexity is the support function
of a convex body. Let $\emptyset\ne K\subseteq\R^{n}$ be a closed,
convex set. Then its support function is a function $h_{K}:\R^{n}\to(-\infty,\infty]$,
defined by 
\[
h_{K}(y)=\sup_{x\in K}\left\langle x,y\right\rangle ,
\]
where $\left\langle \cdot,\cdot\right\rangle $ denotes the standard
Euclidean structure on $\R^{n}$. 

To begin our discussion, we will briefly state a few basic properties
of support functions. A more detailed account, together with all of
the relevant proofs, can be found for example in section 1.7 of \cite{schneider_convex_1993}.
Define 
\[
\Kn=\left\{ \emptyset\ne K\subseteq\R^{n}:\ K\text{ is closed and convex}\right\} .
\]
 For every $K\in\Kn$ its support function $h_{K}$ belongs to 
\[
\Hn=\left\{ \phi:\R^{n}\to(-\infty,\infty]:\ \begin{array}{c}
\phi\text{ is convex, lower semicontinuous and }\\
\text{positively homogenous with }\phi(0)=0
\end{array}\right\} .
\]
 Furthermore, the map $\T:\Kn\to\Hn$ sending $K$ to $h_{K}$ has
the following properties: 
\begin{enumerate}
\item [(P1)]$\mathcal{T}$ is bijective.
\item [(P2)]$\mathcal{T}$ is order preserving: $K_{1}\subseteq K_{2}$
if and only if $\T K_{1}\le\T K_{2}$ (here and after, $f\le g$ means
$f(x)\le g(x)$ for all $x$)
\item [(P3)]$\T$ is additive: For every $K_{1},K_{2}$ we have $\T(K_{1}+K_{2})=\T K_{1}+\T K_{2}$,
where the addition on the left hand side is Minkowski addition:
\[
K_{1}+K_{2}=\left\{ x+y:\ x\in K_{1}\text{ and }y\in K_{2}\right\} .
\]

\end{enumerate}
It turns out that these properties suffice to characterize $\T$ uniquely,
up to a linear change of variables. In fact, one can do even better:
From the work of Gruber in \cite{gruber_endomorphisms_1991} it is
easy to deduce the following:
\begin{thm*}
Assume $\T:\Kn\to\Hn$ satisfies (P1) and (P2) for $n\ge2$. Then
there exists an invertible affine map $B:\R^{n}\to\R^{n}$ such that
$\T\left(K\right)=h_{B(K)}$ for all $K\in\Kn$. 
\end{thm*}
Similarly, In \cite{artstein-avidan_characterization_2010} Artstein-Avidan
and Milman prove:
\begin{thm*}
Assume $\T:\Kn\to\Hn$ satisfies (P1) and (P3) for $n\ge1$. Then
there exists an invertible linear map $B:\R^{n}\to\R^{n}$ such that
$\T\left(K\right)=h_{B(K)}$ for all $K\in\Kn$. 
\end{thm*}
We will now shift our attention from bodies to functions. In recent
years, many notions and results from convexity were generalized from
the class of convex bodies to larger domains. One usual choice for
such a domain is the class of \emph{log-concave} functions. To give
an exact definition, we can define 
\[
\cvx=\left\{ \phi:\R^{n}\to(-\infty,\infty]:\ \begin{array}{c}
\phi\text{ is convex, lower semicontinuous and }\\
\phi(x)<\infty\text{ for some }x
\end{array}\right\} ,
\]
 and then the class of log-concave functions is simply
\[
\lc=\left\{ e^{-\phi}:\ \phi\in\cvx\right\} .
\]
 Put differently, a function $f:\R^{n}\to[0,\infty)$ is log-concave
if $\left(-\log f\right)$ is a convex function. We always assume
our log-concave functions are upper semicontinuous, and explicitly
exclude the constant function $f\equiv0$. There is a natural embedding
of $\Kn$ into $\lc$, which maps every convex body $K$ to its characteristic
function 
\[
\o_{K}=\begin{cases}
1 & x\in K\\
0 & \text{otherwise.}
\end{cases}
\]

Some notions from convexity extend easily to the class of log-concave
functions. For example, since 
\[
\int_{\R^{n}}\o_{K}(x)dx=\text{Vol}(K),
\]
 one can say that the integral of a function $\int f$ extends the
notion of the volume of a convex body $\text{Vol}(K)$. Other extensions
might not be as obvious. It is by now standard to extend the notion
of Minkowski addition by the operation known as sup-convolution, or
Asplund sum: For $f,g\in\lc$ we define their sum by 
\[
\left(f\star g\right)(x)=\sup_{y+z=x}f(y)g(z).
\]
 Notice that this is indeed a generalization, in the sense that $1_{K}\star1_{T}=1_{K+T}$. 

Similarly, it is standard to extend the notion of support function
by defining 
\[
h_{f}=\left(-\log f\right)^{\ast}
\]
 for $f\in\lc.$ The $\ast$ here denotes the classic Legendre transform,
defined by
\[
\phi^{\ast}(y)=\sup_{x\in\R^{n}}\left(\left\langle x,y\right\rangle -\phi(x)\right).
\]
 Again, this is a proper generalization, in the sense that $h_{\o_{K}}=h_{K}$,
as one easily checks.

It turns out that support functions of log-concave functions share
most of the important properties of support functions of convex bodies.
More specifically, the function $\T:\lc\to\cvx$ mapping $f$ to $h_{f}$
has the following properties:
\begin{enumerate}
\item [(Q1)]$\mathcal{T}$ is bijective.
\item [(Q2)]$\mathcal{T}$ is order preserving: $f_{1}\le f_{2}$ if and
only if $\T f_{1}\le\T f_{2}$. 
\item [(Q3)]$\T$ is additive: For every $f_{1},f_{2}$ we have $\T(f_{1}\star f_{2})=\T f_{1}+\T f_{2}$.
\end{enumerate}
Again, one can use properties (Q1)-(Q3) to uniquely characterize the
support function up to a linear change of variables. In \cite{artstein-avidan_concept_2009}
Artstein-Avidan and Milman prove
\begin{thm*}
Assume $\T:\lc\to\cvx$ satisfies (Q1) and (Q2). Then there exists
an invertible affine map $B:\R^{n}\to\R^{n}$ , constants $C_{1},C_{2}$
and a vector $v\in\R^{n}$ such that 
\[
\left(\T f\right)(x)=C_{1}\cdot h_{f}(Bx)+\left\langle x,v\right\rangle +C_{2}
\]
 for all $f\in\lc$.
\end{thm*}
Additionally, in \cite{artstein-avidan_characterization_2010} the
same authors prove the following:
\begin{thm*}
Assume $\T:\lc\to\cvx$ satisfies (Q1) and (Q3). Then there exists
an invertible affine map $B:\R^{n}\to\R^{n}$ and a constant $C$
such that 
\[
\left(\T f\right)(x)=C\cdot h_{f}(Bx)
\]
 for all $f\in\lc$.
\end{thm*}
The last two theorems actually serve an important purpose. The definition
of the support function $h_{K}$ of a convex body exists for a long
time, and has proven itself to be extremely useful. The definition
of the support function $h_{f}$ of a log-concave function, however,
is much newer, and it is reasonable to debate the question of whether
this definition is the ``right'' one. These theorems give us a way
to justify our definitions: If, for example, one believes that (Q1)
and (Q2) are reasonable properties that any definition will have to
satisfy, then the exact definition follows immediately. 

However, as was pointed out by Vitali Milman, the assumption (Q1)
may not be as innocent as it first appears. Injectivity of $\T$ is
fairly natural, and follows easily from property (Q2) as well. Surjectivity,
on the other hand, is a more delicate matter. After all, the support
function $h_{K}$ of a convex body is not an arbitrary convex function,
but always a positively homogenous one. It is possible that $\Im\T$
should also be a proper subset of $\cvx$, and that in order to get
every function $\phi\in\cvx$ as a support function one should increase
the domain $\lc$ even further.

As it turns out, there is a well known way to extend $\lc$ to a larger
class of functions. Consider the following definition:
\begin{defn}
\label{def:a-concave}Fix $-\infty\le\alpha\le\infty$. We say that
a function $f:\R^{n}\to[0,\infty)$ is $\alpha$-concave if $f$ is
supported on some convex set $\Omega$, and for every $x,y\in\Omega$
and $0\le\lambda\le1$ we have
\[
f\left(\lambda x+(1-\lambda)y\right)\ge\left[\lambda f(x)^{\alpha}+\left(1-\lambda\right)f(y)^{\alpha}\right]^{\frac{1}{\alpha}}.
\]
 For $\alpha=-\infty,0,\infty$ we understand this inequality in the
limit sense. This means that $f$ is $\left(-\infty\right)$-concave
if 
\[
f\left(\lambda x+(1-\lambda)y\right)\ge\min\left\{ f(x),f(y)\right\} ,
\]
 is $0$-concave if it is log-concave: 
\[
f\left(\lambda x+(1-\lambda)y\right)\ge f(x)^{\lambda}f(y)^{1-\lambda},
\]
 and $\infty$-concave if 
\[
f\left(\lambda x+(1-\lambda)y\right)\ge\max\left\{ f(x),f(y)\right\} ,
\]
 which implies that $f$ is constant on $\Omega$. 
\end{defn}
In this definition we follow the conventions of Brascamp and Lieb
in \cite{brascamp_extensions_1976}, but the notion can be traced
to the works of Avriel \cite{avriel_r-convex_1972} and Borell \cite{borell_convex_1975}.
The interested reader may also consult \cite{bobkov_convex_2010}
for applications more similar in spirit to this paper.

One of the goals of this paper is to demonstrate how some constructions,
usually carried out for log-concave functions, can also be carried
out for general $\alpha$-concave functions. These constructions will
include the support function, and the mean width. This discussion,
together with a more systematic treatment of $\alpha$-concave functions,
will appear in section \ref{sec:Mean-width}. For now, let us just
mention the fact that if $\alpha_{1}<\alpha_{2}$, then every $\alpha_{2}$-concave
function, is also $\alpha_{1}$-concave. We can use this fact to generate
the following example:
\begin{example}
Fix $-\infty<\alpha<0$. The claim that $f:\R^{n}\to[0,\infty)$ is
$\alpha$-concave is equivalent to saying that $f^{\alpha}$ is convex.
This means we can write every $\alpha$-concave function, and hence
every log-concave function, as 
\[
f(x)=\left[1-\alpha\cdot\phi(x)\right]^{\frac{1}{\alpha}}
\]
for some convex function $\phi$ (notice that $-\alpha$ is positive).
Define $\T_{\alpha}:\lc\to\cvx$ by 
\[
\T_{\alpha}f=\phi^{\ast}.
\]
 It is easy to verify that $\T_{\alpha}$ extends the classic support
function, in the sense that if $f=\o_{K}$ then $\T_{\alpha}f=h_{K}$.
It is equally easy to check that this $\T_{\alpha}$ satisfies property
(Q2), but not property (Q1). Since $\T_{\alpha}f$ is in general very
different from $h_{f}$, we see that the Artstein-Milman characterization
theorem fails completely without the surjectivity assumption. 

To gain an insight into the origins of this example, notice that as
$\alpha\to0$ we have $\phi\to\left(-\log f\right)$, so $\T_{\alpha}f\to h_{f}$,
at least on a heuristic level. Intuitively, one may think of $\T_{\alpha}f$
as the ``right'' definition of the support function of an $\alpha$-concave
function, and the standard definition is just the special case $\alpha=0$.
We will revisit this point of view in section \ref{sec:Mean-width}. 
\end{example}
It is interesting to note that the above example does not satisfy
property (Q3), so at least in this sense it is less natural then the
standard construction. Vitali Milman asked whether it possible to
assume \emph{both} (Q2) and (Q3), and prove a characterization theorem
which does not require surjectivity. Indeed, this is the case:
\begin{thm}
\label{thm:char-lc}Assume we are given a function $\S:\lc\to\cvx$
and an operation $\oplus:\lc\times\lc\to\lc$ with the following properties:
\begin{enumerate}
\item $\S$ is order preserving: $f\le g$ if and only if $\S f\le\S g$.
\item $\S$ extends the usual support functional: If $f=\o_{K}$, then $\S f=h_{K}$.
\item $\S\left(f\oplus g\right)=\S f+\S g$.
\end{enumerate}
Then we must have
\[
\left(\S f\right)(x)=\frac{1}{C}h_{f}(C\cdot x)
\]
 for some $C>0$, and $f\oplus g=f\star g$. 
\end{thm}
Note that even though we do not assume surjectivity a priori, it follows
a posteriori that $\S$ must be onto $\cvx$. Another interesting
feature of the theorem is that our third assumption is somewhat weaker
than (Q3), as we only need to assume that $S$ is additive with respect
to \emph{some }addition operation $\oplus$. Therefore this theorem
characterizes not only the support function, but the sup-convolution
operation as well. 

Theorem \ref{thm:char-lc} will follow easily from the following theorem:
\begin{thm}
\label{thm:char-cvx}Assume a function $\T:\cvx\to\cvx$ satisfies
the following:
\begin{enumerate}
\item $\T$ is order preserving: $\phi\le\psi$ if and only if $\T\phi\le\T\psi$.
\item If $\phi$ is a positively homogenous function then $\T\phi=\phi.$
\item The set $\Im\T=\left\{ \T\phi:\ \phi\in\cvx\right\} $ is closed under
pointwise addition.
\end{enumerate}
Then $\left(\T\phi\right)(x)=\frac{1}{C}\phi(Cx)$ for some $C>0$. \end{thm}
\begin{proof}
[Proof of the reduction]Assume theorem \ref{thm:char-cvx} holds.
In order to prove theorem \ref{thm:char-lc}, choose $\S:\lc\to\cvx$
satisfying all of the assumptions. Define $\T:\cvx\to\cvx$ by 
\[
\T\phi=\S\left(e^{-\phi^{\ast}}\right).
\]
 It it easy to see that $\T$ satisfies all of the assumptions of
theorem \ref{thm:char-cvx}, so 
\[
\left(\S\left(e^{-\phi^{\ast}}\right)\right)(x)=\left(\T\phi\right)(x)=\frac{1}{C}\phi(Cx)
\]
 for some $C>0$. For every $f\in\lc$ define $\phi=h_{f}$, and notice
that $e^{-\phi^{\ast}}=f$. Hence we get 
\[
\left(Sf\right)(x)=\frac{1}{C}h_{f}(Cx)
\]
 like we wanted. In particular, for every $f,g\in\lc$ we will get
\[
S(f\star g)=Sf+Sg=S(f\oplus g),
\]
 and since $\S$ is injective $f\oplus g=f\star g$. This completes
the proof.
\end{proof}
The rest of this paper is organized as follows: Section \ref{sec:proof}
will include the proof of theorem \ref{thm:char-cvx}. The proof is
rather long, and is composed of several independent ingredients. Some
of the ingredients are fairly standard by now, but some are probably
new. Let me thank Alexander Segal and Boaz Slomka for providing some
of these arguments. Section \ref{sec:Mean-width} will be devoted
to $\alpha$-concave functions. We will extend the notions discussed
in this section, specifically addition and support functions, to the
realm of $\alpha$-concave functions. Finally, we will define the
mean width of an $\alpha$-concave function, and generalize some known
results to the new setting.

\section{\label{sec:proof}Proof of theorem \ref{thm:char-cvx}}

We will now prove Theorem \ref{thm:char-cvx}. As the proof is quite
long, we will divide it into several parts:

\subsection*{1. ``Negatively affine'' functions}

Fix a vector $a\in\R^{n}$. The linear function $\ell(x)=\left\langle x,a\right\rangle $
is 1-homogenous, so $\T\ell=\ell$. We will now deal with affine functions
of the form 
\[
\ell_{1}(x)=\left\langle x,a\right\rangle +c
\]
 for some $c<0$ (we call such functions ``negatively affine functions'').
Since $\ell_{1}\le\ell$ and $\T$ is order preserving we must have
$\T\ell_{1}\le\T\ell=\ell$, and it follows immediately that 
\[
\left(\T\ell_{1}\right)(x)=\left\langle x,a\right\rangle +c'
\]
 for some constant $c'<0$. 

Later we will show that $\Im\T$ contains all affine functions. For
now, just notice that if $n\in\N$ then we can write
\[
\left\langle x,a\right\rangle +nc'=n\left[\left\langle x,a\right\rangle +c'\right]+\left\langle x,a-na\right\rangle ,
\]
 and since $\Im\T$ is closed under addition we get that $\left\langle x,a\right\rangle +nc'\in\Im\T$.
In particular, there exists a sequence $c_{n}\to-\infty$ such that
$\left\langle x,a\right\rangle +c_{n}\in\Im\T$. 

Finally, notice that if a negatively affine function 
\[
\ell'(x)=\left\langle x,a\right\rangle +c'
\]
 is in $\Im\T$, then we can apply the same reasoning as above for
$\T^{-1}$ and conclude that $\left(\T^{-1}\ell'\right)(x)=\left\langle x,a\right\rangle +c$
for some $c<0$.

\subsection*{2. ``Positively affine'' functions}

Assume now that we are given a function $\ell(x)=\left\langle x,a\right\rangle +c$
with $a\in\R^{n}$ and $c>0$. Define $\phi'=\T\ell$, and let 
\[
\ell_{1}^{\prime}(x)=\left\langle x,b\right\rangle +d_{1}^{\prime}
\]
 be any tangent to $\phi'$. As we saw before, we can choose a constant
$d_{2}^{\prime}\le\min\left(d_{1}^{\prime},0\right)$ such that 
\[
\ell_{2}^{\prime}(x)=\left\langle x,b\right\rangle +d_{2}^{\prime}\in\Im\T.
\]
 Since $\ell_{2}^{\prime}\le\ell_{1}^{\prime}\le\phi'=\T\ell$, it
follows that $\ell_{2}\le\ell$, where
\[
\ell_{2}(x)=\left(\T^{-1}\ell_{2}^{\prime}\right)(x)=\left\langle x,b\right\rangle +d_{2},
\]
 and then we must have $b=a$. In other words, every tangent to $\phi'$
is of the form $\left\langle x,a\right\rangle +d$ for some $d$,
and this can only happen if 
\[
\phi'(x)=\left\langle x,a\right\rangle +c'
\]
 for some $c'$. Since $\ell(x)\ge\left\langle x,a\right\rangle $
we of course have $c'>0$. 

Again, if $\ell'(x)=\left\langle x,a\right\rangle +c'$ for $c'>0$
happens to be in $\Im\T$, we can repeat the argument in reverse and
conclude that 
\[
\left(\T^{-1}\ell'\right)=\left\langle x,a\right\rangle +c
\]
 for some $c>0$.

\subsection*{3. Surjectivity on affine functions}

So far we have seen that that image of any affine function is affine.
We will now show that all affine functions are in $\Im\T$. Fix $a\in\R^{n}$,
and define $f_{a}:\R\to\R$ according to the formula
\[
\T\left(\left\langle x,a\right\rangle +c\right)=\left\langle x,a\right\rangle +f_{a}(c).
\]
Also define 
\[
H_{a}=\text{Im}(f_{a})=\left\{ c'\in\R:\ \left\langle x,a\right\rangle +c'\in\Im\T\right\} .
\]
 Notice that $H_{a}$ is closed under addition: If $c_{1},c_{2}\in H_{a}$
then 
\[
\left\langle x,a\right\rangle +\left(c_{1}+c_{2}\right)=\left[\left\langle x,a\right\rangle +c_{1}\right]+\left[\left\langle x,a\right\rangle +c_{2}\right]+\left\langle x,-a\right\rangle \in\Im\left(\T\right),
\]
 so $c_{1}+c_{2}\in H_{a}$ as well. We will now apply the following
Lemma:
\begin{lem}
\label{lem:submonoids}Assume $H\subseteq\R$ is a subset with the
following properties:
\begin{enumerate}
\item If $x,y\in H$ then $x+y\in H$.
\item There exists $x\in H$ such that $x>0$.
\item There exists $x\in H$ such that $x<0$.
\end{enumerate}
Then $H$ is either a cyclic subgroup of $\R$ or dense in $\R$.
\end{lem}
This result, or slight variations thereof, is well known in some fields.
For the sake of completeness, we will prove Lemma \ref{lem:submonoids}
after we finish proving Theorem \ref{thm:char-cvx}. 

By the above discussion we see that $H_{a}$ satisfies the hypotheses
of Lemma \ref{lem:submonoids}, so it is either cyclic or dense. Since
$f_{a}$ is injective, $H_{a}$ must is uncountable, and since cyclic
groups are all countable, $H_{a}$ must be dense. Now $f_{a}:\R\to\R$
is a monotone function with dense image, and it is an easy exercise
that all such functions are onto. Therefore $H_{a}=\R$ and we proved
the desired result.

\subsection*{4. Delta functions }

We will return to affine functions shortly, but before we do we need
to discuss delta functions. For $a\in\R^{n}$ and $c\in\R$ define
\[
\delta_{a,c}(x)=\begin{cases}
c & x=a\\
\infty & \text{otherwise}.
\end{cases}
\]
 Our goal is to show that delta functions are mapped to delta functions.
Assume by contradiction that $\phi'=\T\delta_{a,c}$ is not a delta
function, so there exists $b_{1}\ne b_{2}$ such that $\phi'(b_{1}),\phi'(b_{2})<\infty$.
We will divide the proof into two cases:
\begin{caseenv}
\item There exists a constant $\lambda>0$ such that $\lambda b_{1}=b_{2}$.
Without loss of generality we can assume $\lambda>1$ (otherwise,
swap $b_{1}$ and $b_{2}$ in the following proof). Define two functions
\begin{eqnarray*}
\psi_{1}^{\prime}(x) & = & \left|x\right|\\
\psi_{2}^{\prime}(x) & = & \frac{1+\lambda}{2}\left|b_{1}\right|
\end{eqnarray*}
(where$\left|\cdot\right|$ denotes the euclidean norm). $\psi_{1}^{\prime}\in\Im\T$
as a positively homogenous function, and $\psi_{2}^{\prime}\in\Im\T$
as a constant, hence affine, function. Therefore if we define $\rho_{i}^{\prime}=\phi^{\prime}+\psi_{i}^{\prime}$
for $i=1,2$ then $\rho_{i}^{\prime}\in\Im\T$ as well. Define $\rho_{i}=\T^{-1}(\rho_{i}^{\prime})$.
Since $\rho_{i}^{\prime}\ge\phi^{\prime}=\T\delta_{a,c}$, we have
$\rho_{i}\ge\delta_{a,c}$, so $\rho_{i}=\delta_{a,c_{i}}$ for some
$c_{i}$. In particular, we must have $\rho_{1}\ge\rho_{2}$ or vice
versa. But this implies that $\rho_{1}^{\prime},\rho_{2}^{\prime}$
are comparable as well, which is a contradiction:
\begin{eqnarray*}
\rho_{1}^{\prime}(b_{1})=\phi'(b_{1})+\left|b_{1}\right| & < & \phi'(b_{1})+\frac{1+\lambda}{2}\left|b_{1}\right|=\rho_{2}^{\prime}(b_{1})\\
\rho_{1}^{\prime}(b_{2})=\phi'(b_{2})+\lambda\left|b_{1}\right| & > & \phi'(b_{2})+\frac{1+\lambda}{2}\left|b_{1}\right|=\rho_{2}^{\prime}(b_{2}).
\end{eqnarray*}

\item Now assume $b_{1}$ and $b_{2}$ are not on the same ray. In this
case define
\begin{eqnarray*}
\psi_{1}^{\prime}(x) & = & \begin{cases}
0 & x\in\R^{+}b_{1}\\
\infty & \text{otherwise}
\end{cases}\\
\psi_{2}^{\prime}(x) & = & 1,
\end{eqnarray*}
 and the rest of the proof is exactly the same as the previous case.
\end{caseenv}
In both cases we arrived at a contradiction, so we proved that indeed
$\phi'$ is a delta function like we wanted.

\subsection*{5. Tangents}

Let $\ell$ be an affine function and $\phi\in\cvx$ be arbitrary.
We say that $\ell$ is tangent to $\phi$ if $\ell\le\phi$, but $\ell+\epsilon\not\le\phi$
for every $\epsilon>0$. It is well known that 
\[
\phi=\sup\left\{ \ell:\ \ell\text{ is tangent to }\phi\right\} .
\]

Our simple claim is that $\ell$ is tangent to $\phi$ if and only
if $\ell'=\T\ell$ is tangent to $\phi'=\T\phi$. Indeed, if $\ell$
is tangent to $\phi$, then we immediately get $\ell'\le\phi'$. If
$\ell'+\epsilon\le\phi'$ for some $\epsilon>0$ then $\T^{-1}\left(\ell'+\epsilon\right)$
is an affine function such that 
\[
\ell<\T^{-1}\left(\ell'+\epsilon\right)\le\phi,
\]
 which is impossible. The other direction is proven in exactly the
same way.

In particular, if $\phi=\delta_{a,c}$, then $\ell$ is tangent to
$\phi$ if and only if $\ell(a)=c$. Therefore if $\ell$ is an affine
function and $a\in\R^{n}$ , we can always find $b\in\R^{n}$ such
that
\[
\T\left(\delta_{a,\ell(a)}\right)=\delta_{b,\left(\T\ell\right)(b)}.
\]

\subsection*{6. Collinearity}

We will identify every affine map $\ell(x)=\left\langle x,a\right\rangle +c$
with the point $p_{\ell}=(a,c)\in\R^{n+1}$. Under this identification
our map $\T$ induces a bijection $F:\R^{n+1}\to\R^{n+1}$ defined
by 
\[
F(p_{\ell})=p_{\T\ell}.
\]
 Our current goal is to show that $F^{-1}$ maps collinear points
into collinear points. The following lemma will prove itself useful:
\begin{lem}
\label{lem:collinearity}Assume $\ell_{1},\ell_{2},\ell_{3}$ are
3 affine functions such that $\ell_{1}$ and $\ell_{2}$ are not parallel.
Then $p_{\ell_{1}},p_{\ell_{2}},p_{\ell_{3}}$ are collinear if and
only if whenever $\ell_{1}(x_{0})=\ell_{2}(x_{0})$ we also have $\ell_{1}(x_{0})=\ell_{3}(x_{0})$. 
\end{lem}
Again, we will postpone the proof of Lemma \ref{lem:collinearity}
until the end of the proof of Theorem \ref{thm:char-cvx}. For now
we will use this lemma to prove the result: Assume $\ell_{1}^{\prime},\ell_{2}^{\prime},\ell_{3}^{\prime}$
are collinear affine functions, and define $\ell_{i}=\mathcal{T}^{-1}\ell_{i}^{\prime}$.
If $\ell_{1}^{\prime}$ and $\ell_{2}^{\prime}$ are parallel then
all six functions are parallel to each other and there is nothing
to prove.

In the general case, assume $a$ is any point such that
\[
\ell_{1}(a)=\ell_{2}(a)=c,
\]
 and define $\phi=\delta_{a,c}$. We've already seen that in this
case we must have
\[
\T\phi=\delta_{b,d},
\]
 where $\ell_{1}^{\prime}(b)=d=\ell_{2}^{\prime}(b)$. Since $\ell_{1}^{\prime},\ell_{2}^{\prime},\ell_{3}^{\prime}$
are collinear we get from Lemma \ref{lem:collinearity} that $\ell_{3}^{\prime}(b)=d$
as well. This implies that $\delta_{b,d}$ is tangent to $\ell_{3}^{\prime}$,
so $\delta_{a,c}$ is tangent to $\ell_{3}$ and $\ell_{3}(a)=c$.
Again by Lemma \ref{lem:collinearity} we get that $\ell_{1},\ell_{2},\ell_{3}$
are collinear like we wanted.

By the fundamental theorem of affine geometry, it now follows that
$F^{-1}$ is affine, so $F$ is affine as well (for an exact formulation
of the fundamental theorem and a sketch of the proof, the reader may
consult \cite{artstein-avidan_concept_2009}). This means that we
can write
\[
\T\left(\left\langle x,a\right\rangle +c\right)=\left\langle x,a\right\rangle +\left\langle a,v\right\rangle +\gamma c
\]
 for some constants $\gamma\in\R^{n}$ and $v\in\R^{n}$ (which are
of course independent of $a$ and $c$). 

We know that $\T\left(\left\langle x,a\right\rangle \right)=\left\langle x,a\right\rangle $,
so $\left\langle a,v\right\rangle =0$ for all $a\in\R^{n}$, which
implies $v=0$. Also, for $\T$ to be order preserving, we must have
$\gamma>0$.

\subsection*{7. Finishing the proof}

We now know that 
\[
\T\left(\left\langle x,a\right\rangle +c\right)=\left\langle x,a\right\rangle +\gamma c
\]
 for some $\gamma>0$. Remember that in the statement of Theorem \ref{thm:char-cvx}
we had one degree of freedom - we don't want to prove that $\T\phi=\phi$,
but that $\left(\T\phi\right)(x)=\frac{1}{C}\phi(Cx)$ for some $C>0$.
We will now use this degree of freedom and assume that $\gamma=1$
(formally, this means we replace $\T$ with $\widetilde{\T}$ defined
by $\left(\widetilde{\T}\phi\right)(x)=\frac{1}{\gamma}\left(\T\phi\right)(\gamma x)$.
We will keep using the notation$\T$ for the new function). 

For every function $\phi\in\cvx$ and any affine $\ell$, we know
that $\ell$ is tangent to $\phi$ if and only if $\T\ell=\ell$ is
tangent to $\T\phi$. In other words, $\phi$ and $\T\phi$ have exactly
the same tangents, so $\T\phi=\phi$ and our proof is finally complete.

\subsection*{8. Proofs of the lemmas}
\begin{proof}
[Proof of Lemma \ref{lem:submonoids}.]First we note that it is enough
to prove that $H$ is either cyclic or dense in $\R^{+}=[0,\infty)$.
Indeed, assume that $H$ is dense in $\R^{+}$. We know that there
exists an element $x<0$ in $H$. If $y\in\R$ is any number, we can
choose $n\in\N$ so large that $y-nx>0$. Now we can find a sequence
$\left\{ x_{k}\right\} \subseteq H$ such that $x_{k}\to y-nx$, and
then $x_{k}+nx\to y$, so $H$ is dense in $\R$. Of course, by a
symmetric argument, it is also enough to prove that $H$ is dense
in $\R^{-}=(-\infty,0]$. 

Now we define 
\begin{eqnarray*}
a & = & \inf\left\{ x\in H:\ x>0\right\} \\
b & = & \sup\left\{ x\in H:\ x<0\right\} .
\end{eqnarray*}

If $a=0$ then there exists a sequence $\left\{ x_{k}\right\} \subseteq H$
such that $x_{k}>0$ for all $k$ and $x_{k}\to0$. But then the set
$\left\{ n\cdot x_{k}:\ n,k\in\N\right\} \subseteq H$ is dense in
$\R^{+}$, so we are done. Similarly, if $b=0$ then $H$ is dense
in $\R^{-}$ and we are done as well. Hence we will assume that $b<0<a$,
and prove that $H$ is cyclic.

Our next goal is to prove that $\left|a\right|=\left|b\right|$. If
not, we may assume without loss of generality that $\left|a\right|>\left|b\right|$,
or, put differently, $a+b>0$. Choose sequences $\left\{ x_{k}\right\} ,\left\{ y_{k}\right\} \subseteq H$
such that $x_{k}\to a$ and $y_{k}\to b$. Then $x_{k}+y_{k}\to a+b$.
Since $0<a+b<a$, for large enough $k$ we have $0<x_{k}+y_{k}<a$,
which is a contradiction to the definition of $a$. Therefore $b=-a$
like we wanted.

Now we prove that $a\in H$. If $a\notin H$, then for every $\epsilon>0$
one can find an element $x\in H$ such that $a<x<a+\epsilon$. In
particular, we can choose $x$ such that $a<x<2a$. Like before, choose
a sequence $\left\{ y_{k}\right\} \subseteq H$ such that $y_{k}\to b=-a$.
Then $x+y_{k}\to x-a$, and in particular for large enough $k$ we
have $0<x+y_{k}<a$. Again, this is a contradiction to the definition
of $a$, so $a\in H$. An identical argument shows that $-a\in H$
as well.

To conclude the proof we finally show that $H=\left\langle a\right\rangle =\left\{ n\cdot a:\ n\in\mathbb{Z}\right\} $.
The fact that $\left\langle a\right\rangle \subseteq H$ is now obvious.
For the other direction, every element $x\in\R^{+}$ can be written
as $x=n\cdot a+y$ for $n\in\N$ and $0\le y<a$. If $x\in H$ then
\[
y=x-n\cdot a=x+n\cdot\left(-a\right)\in H
\]
 as well, so by the minimality of $a$ we must have $y=0$. Therefore
$x=na$, so $x\in\left\langle a\right\rangle $. A similar argument
works in the case $x\in\R^{-}$,and the proof is complete.
\end{proof}

\begin{proof}
[Proof of Lemma \ref{lem:collinearity}.] One of the implications
is simple: If $p_{\ell_{1}},p_{\ell_{2}},p_{\ell_{3}}$ are collinear
then we can write
\[
p_{\ell_{3}}=\lambda p_{\ell_{1}}+(1-\lambda)p_{\ell_{2}}
\]
for some $\lambda\in\R$. This implies that 
\[
\ell_{3}=\lambda\ell_{1}+(1-\lambda)\ell_{2},
\]
 and then $\ell_{1}(x_{0})=\ell_{2}(x_{0})=c$ implies $\ell_{3}(x_{0})=\lambda c+(1-\lambda)c=c$
as well.

The other implication is almost as easy. Write
\[
\ell_{i}\left(x\right)=\left\langle x,a_{i}\right\rangle +c_{i}
\]
 for $i=1,2,3$. Our assumption can be reformulated as saying that
$\left\langle x,a_{1}-a_{2}\right\rangle =c_{2}-c_{1}$ implies $\left\langle x,a_{1}-a_{3}\right\rangle =c_{3}-c_{1}.$
By standard linear algebra, this can only happen if there exists a
$\lambda\in\R$ such that
\begin{eqnarray*}
a_{1}-a_{3} & = & \lambda\left(a_{1}-a_{2}\right)\\
c_{3}-c_{1} & = & \lambda\left(c_{2}-c_{1}\right).
\end{eqnarray*}
 This is equivalent to
\[
\ell_{3}=\left(1-\lambda\right)\ell_{1}+\lambda\ell_{2},
\]
 which implies collinearity of $p_{\ell_{1}},p_{\ell_{2}},p_{\ell_{3}}$
like we wanted.
\end{proof}

\section{\label{sec:Mean-width}Mean width for $\alpha$-concave functions}

We will now begin our discussion of $\alpha$-concave functions (see
Definition \ref{def:a-concave}). For simplicity, we will restrict
ourselves to case $-\infty<\alpha\le0$. For any such $\alpha$, define
\[
C_{\alpha}\left(\R^{n}\right)=\left\{ f:\R^{n}\to[0,\infty):\ f\text{ is }\mbox{\ensuremath{\alpha}}\text{-concave, upper semicontinuous and }f\not\equiv0\right\} .
\]
 For example, we have $C_{0}(\R^{n})=\lc$. As stated in section \ref{sec:support-functions},
we have $C_{\alpha_{1}}(\R^{n})\subseteq C_{\alpha_{2}}(\R^{n})$
whenever $\alpha_{1}\ge\alpha_{2}$.
\begin{rem}
\label{rem:k-concave-measures}In \cite{borell_convex_1975}, Borell
defines not only $\alpha$-concave functions, but also the notion
of a $\kappa$-concave measure. A Radon measure $\mu$ on $\R^{n}$
is $\kappa$-concave if for any non empty Borel sets $A,B$ and any
$0<\lambda<1$ we have 
\[
\mu\left(\lambda A+(1-\lambda)B\right)\ge\left[\lambda\mu(A)^{\kappa}+\left(1-\lambda\right)\mu(B)^{\kappa}\right]^{\frac{1}{\kappa}}.
\]
 Borell then proves that $\alpha$-concave functions and $\kappa$-concave
measures are closely related: Assume $\mu$ is not support on any
hyperplane. Then $\mu$ is $\kappa$-concave if and only if $\kappa\le\frac{1}{n}$,
$\mu$ is absolutely continuous with respect to the Lebesgue measure,
and the density $f=\frac{d\mu}{dx}$ is $\alpha$-concave, for $\alpha=\frac{\kappa}{1-n\kappa}$. 

Notice that for such a density $f=\frac{d\mu}{dx}$, we must have
$\alpha\ge-\frac{1}{n}$, so some authors only discuss $\alpha$-concave
functions for such values of $\alpha$. We will need the assumption
$\alpha\ge-\frac{1}{n}$ for some theorems, but other results will
hold in full generality.
\end{rem}
Since we only care about negative values of $\alpha$, it will often
be more convenient, and less confusing, to use the parameter $\beta=-\frac{1}{\alpha}$.
For example, we will use the new notation in the following definition:
\begin{defn}
The convex base of the a function $f\in\Ca$ is 
\[
\base(f)=\frac{1-f^{\alpha}}{\alpha}.
\]
 Put differently, $\phi=\base(f)$ is the unique convex function such
that 
\[
f=\left(1+\frac{\phi}{\beta}\right)^{-\beta}.
\]

\end{defn}
The above definition is inspired by the work of Bobkov in \cite{bobkov_convex_2010}.
While the definition might seem unintuitive at first, it has a couple
of appealing features :
\begin{itemize}
\item In the limiting case $\alpha\to0$ ($\beta\to\infty$), we get the
relation $\text{base}_{0}f=-\log f$. This is the standard and often
used bijective, order reversing map between $\lc$ and $\cvx$.
\item If $f=\o_{K}$ is an indicator function, then 
\[
\base f=\o_{K}^{\infty}=\begin{cases}
0 & x\in K\\
\infty & \text{otherwise}
\end{cases}
\]
is the well known ``convex indicator function'' of $K$. In particular,
$\base f$ is independent of $\alpha$ in that case.
\end{itemize}
Notice, however, that unlike the log-concave case, the map $\base:\Ca\to\cvx$
is not surjective, as we always have $\base f>-\beta$. 

If we are willing to treat $\base f$ as the proper generalization
of $\left(-\log f\right)$ to the $\alpha$-concave case, a few important
definitions emerge immediately:
\begin{defn}

\begin{enumerate}
\item The support function of a function $f\in\Ca$ is 
\[
h_{f}^{\left(\alpha\right)}=\left(\base f\right)^{\ast}\in\cvx.
\]

\item The sum of two functions $f,g\in\Ca$ is defined by
\[
\base\left(f\star_{\alpha}g\right)=\left(\base f\right)\square\left(\base g\right),
\]
assuming the right hand side is a indeed a convex base for an $\alpha$-concave
function (see the discussion above Proposition \ref{pro:conv-repr}).
Here $\square:\cvx\times\cvx\to\cvx$ is the standard inf-convolution,
defined by 
\[
\left(\phi\square\psi\right)(x)=\inf_{y+z=x}\left[\phi(y)+\psi(z)\right].
\]

\item If $f\in\Ca$ and $\lambda>0$, the $\lambda$-homothety of $f$ defined
by 
\[
\left[\base\left(\lambda\cdot_{\alpha}f\right)\right](x)=\lambda\cdot\left(\base f\right)\left(\frac{x}{\lambda}\right).
\]

\end{enumerate}
When there is no cause for confusion, we will omit the script and
write $h_{f}$, $f\star g$, and $\lambda\cdot f$. 
\end{defn}
The above definitions were constructed to interact well with one another.
We have for example 
\[
h_{\left(\lambda\cdot_{\alpha}f\right)\star_{\alpha}g}^{\left(\alpha\right)}=\lambda h_{f}^{\left(\alpha\right)}+h_{g}^{\left(\alpha\right)},
\]
as well as $f\star_{\alpha}f=2\cdot_{\alpha}f$ and other similar
equalities. However, one should be aware of two important caveats.

The first thing to observe is that the above definitions really depend
on $\alpha$. We know that if $f\in C_{\alpha_{1}}(\R^{n})$, then
$f\in C_{\alpha_{2}}(\R^{n})$ for every $\alpha_{2}<\alpha_{1}$.
Nonetheless, we usually have $h_{f}^{\left(\alpha_{1}\right)}\ne h_{f}^{\left(\alpha_{2}\right)}$,
and similarly for additions and homotheties. An important exception
to this rule is the case of indicator functions. If $f=\o_{K}$ and
$g=\o_{T}$, then $h_{f}^{\left(\alpha\right)}=h_{K}$, $f\star_{\alpha}g=\o_{K+T}$
and $\lambda\cdot_{\alpha}f=\o_{\lambda K}$, for all values of $\alpha$. 

The second, technical, caveat is that additions and homotheties are
not always defined. If, for example, $\base f=\base g=-\frac{3\beta}{4}$,
then 
\[
\base\left(f\star_{\alpha}g\right)=\left(\base f\right)\square\left(\base g\right)=-\frac{3\beta}{2}.
\]
But this is impossible, since for every $h\in\Ca$ we have $\base h\ge-\beta$.
Addition is defined, however, under some mild conditions on $f$ and
$g$ (for example it is enough to assume $f\le1$). A particularly
nice case is the case of convex combinations, where we have the following
simple formula:
\begin{prop}
\label{pro:conv-repr}Fix $f,g\in\Ca$ and choose $0<\lambda<1$.
Define 
\[
h=\left[\lambda\cdot f\right]\star\left[\left(1-\lambda\right)\cdot g\right].
\]
 Then 
\[
h(x)=\sup_{y+z=x}\left[\lambda f\left(\frac{y}{\lambda}\right)^{\alpha}+(1-\lambda)\cdot g\left(\frac{z}{1-\lambda}\right)^{\alpha}\right]^{\frac{1}{\alpha}}
\]
\end{prop}
\begin{proof}
This is nothing more than an explicit calculation. Denote $\phi=\ker_{\alpha}f$
and $\psi=\ker_{\alpha}g$. Then: 
\begin{eqnarray*}
h(x)^{\alpha} & = & \left(1+\frac{\ker_{\alpha}h}{\beta}\right)^{-\beta\alpha}\\
 & = & 1+\inf_{y+z=x}\frac{\lambda\phi\left(\frac{y}{\lambda}\right)+(1-\lambda)\psi\left(\frac{z}{1-\lambda}\right)}{\beta}\\
 & = & \inf_{y+z=x}\left[\lambda\cdot\left(1+\frac{\phi\left(\frac{y}{\lambda}\right)}{\beta}\right)+\left(1-\lambda\right)\left(1+\frac{\psi\left(\frac{z}{1-\lambda}\right)}{\beta}\right)\right]\\
 & = & \inf_{y+z=x}\left[\lambda f\left(\frac{y}{\lambda}\right)^{\alpha}+(1-\lambda)\cdot g\left(\frac{z}{1-\lambda}\right)^{\alpha}\right],
\end{eqnarray*}
 and raising both sides to power $\frac{1}{\alpha}$ we get the result.
\end{proof}
Proposition \ref{pro:conv-repr} is especially useful when combined
with a known inequality, discovered independently by Borell, (\cite{borell_convex_1975})
and Brascamp and Lieb (\cite{brascamp_extensions_1976}): 
\begin{thm*}
[Borell-Brascamp-Lieb]\label{thm:BBL}Assume we are given measurable
functions $f,g,h:\R^{n}\to[0,\infty]$ and numbers $0<\lambda<1$,
$\alpha\ge-\frac{1}{n}$ such that 
\[
h\left(\lambda x+(1-\lambda)y\right)\ge\left[\lambda f(x)^{\alpha}+\left(1-\lambda\right)g(y)^{\alpha}\right]^{\frac{1}{\alpha}}
\]
 whenever $f(x),g(y)>0$. Then
\[
\int h\ge\left[\lambda\left(\int f\right)^{\kappa}+\left(1-\lambda\right)\left(\int g\right)^{\kappa}\right]^{\frac{1}{\kappa}},
\]
 where $\kappa=\frac{\alpha}{1+n\alpha}$. 
\end{thm*}
The important of the parameter $\kappa$ was explained in Remark \ref{rem:k-concave-measures}.
Notice that when $\alpha=\infty$ we get $\kappa=\frac{1}{n}$ and
the theorem reduces to the Brunn-Minkowski theorem. When $\alpha=0$
we get that $\kappa=0$ as well and the theorem reduces to the a special
case known as the Prékopa\textendash{}Leindler inequality. 

From Proposition \ref{pro:conv-repr} and Theorem \ref{thm:BBL} we
immediately get: 
\begin{cor}
\label{cor:BBL-modified}If $f,g\in\Ca$ and $\alpha\ge-\frac{1}{n}$,
then 
\[
\int\left[\lambda\cdot f\right]\star\left[\left(1-\lambda\right)\cdot g\right]\ge\left[\lambda\left(\int f\right)^{\kappa}+\left(1-\lambda\right)\left(\int g\right)^{\kappa}\right]^{\frac{1}{\kappa}}.
\]

\end{cor}
Our next goal is to define the mean width of an $\alpha$-concave
function. For log-concave functions, the concept of mean width was
originally defined by Klartag and Milman in \cite{klartag_geometry_2005}.
If $f\in\lc$, the Klartag-Milman definition for the mean width of
$f$ is, up to some universal constant, 
\[
w(f)=\lim_{\epsilon\to0^{+}}\frac{\int G\star\left[\epsilon\cdot f\right]-\int G}{\epsilon},
\]
 where $G(x)=e^{-\left|x\right|^{2}/2}$ is the (unnormalized) Gaussian.
Since we are dealing with log-concave functions, $\star$ means $\star_{0}$
in our notation. In \cite{rotem_mean_2012}, the author presented
an equivalent definition, as the average of the support function with
respect to the Gaussian measure:
\[
w(f)=\int_{\R^{n}}h_{f}(x)\cdot G(x)dx.
\]

Both definitions can be extended, mutatis mutandis, to general $\alpha$-concave
functions.
\begin{defn}

\begin{enumerate}
\item For $-\infty<\alpha\le0$ define a function $G_{\alpha}\in\Ca$ by
\[
G_{\alpha}(x)=\left(1+\frac{\left|x\right|^{2}}{2\beta}\right)^{-\beta},
\]
 where, as usual $\beta=-\frac{1}{\alpha}$. In other words, we choose
$G_{\alpha}$ to satisfy $\base G_{\alpha}=\frac{\left|x\right|^{2}}{2}$
.
\item For $f\in\Ca$ we define its $\alpha$- mean width as
\[
w_{\alpha}(f)=\lim_{\epsilon\to0^{+}}\frac{\int G_{\alpha}\star_{\alpha}\left[\epsilon\cdot_{\alpha}f\right]-\int G_{\alpha}}{\epsilon}
\]

\end{enumerate}
\end{defn}
The results of \cite{rotem_mean_2012} can be extended to our case
as well. For example we have the following representation formula:
\begin{thm}
\label{thm:width-repr}For every $f\in\Ca$ we get 
\[
w_{\alpha}(f)=\int h_{f}^{\left(\alpha\right)}(x)\cdot\left(1+\frac{\left|x\right|^{2}}{2\beta}\right)^{-\beta-1}dx
\]
\end{thm}
\begin{proof}
A considerable amount of work is needed in order to write down a completely
formal proof, which applies to all cases. All of the details appear
in \cite{rotem_mean_2012} for the log-concave case, but the same
strategy works just as well for the general $\alpha$-concave case.
Here we give the essence of the proof, and the suspicious reader may
consult \cite{rotem_mean_2012} for the finer details:

Denote $\phi=\base f$. Then
\begin{eqnarray*}
\base\left(G_{\alpha}\star\left(\epsilon\cdot f\right)\right)(x) & = & \inf_{y}\left[\frac{\left|x-y\right|^{2}}{2}+\epsilon\phi\left(\frac{y}{\epsilon}\right)\right]\\
 & = & \frac{\left|x\right|^{2}}{2}+\inf_{y}\left[\frac{\left|y\right|^{2}}{2}-\left\langle x,y\right\rangle +\epsilon\phi\left(\frac{y}{\epsilon}\right)\right]\\
 & \underset{y=\epsilon z}{=} & \frac{\left|x\right|^{2}}{2}+\inf_{z}\left[\frac{\left|\epsilon z\right|^{2}}{2}-\left\langle x,\epsilon z\right\rangle +\epsilon\phi\left(z\right)\right]\\
 & = & \frac{\left|x\right|^{2}}{2}+\epsilon\cdot\inf_{z}\left[\epsilon\frac{\left|z\right|^{2}}{2}-\left\langle x,z\right\rangle +\phi\left(z\right)\right]\\
 & = & \frac{\left|x\right|^{2}}{2}-\epsilon\cdot\sup_{z}\left[\left\langle x,z\right\rangle -\left(\phi\left(z\right)+\epsilon\frac{\left|z\right|^{2}}{2}\right)\right]\\
 & = & \frac{\left|x\right|^{2}}{2}-\epsilon\cdot\left(\phi+\epsilon\frac{\left|x\right|^{2}}{2}\right)^{\ast}(x).
\end{eqnarray*}

Define
\[
H(x,\epsilon)=\frac{\left|x\right|^{2}}{2}-\epsilon\cdot\left(\phi+\epsilon\frac{\left|x\right|^{2}}{2}\right)^{\ast}(x),
\]
 then by the product rule
\begin{eqnarray*}
\left.\frac{dH}{d\epsilon}\right|_{\epsilon=0} & = & -\left.\left(\phi+\epsilon\frac{\left|x\right|^{2}}{2}\right)^{\ast}(x)\right|_{\epsilon=0}-0\cdot\left[\left.\frac{d}{d\epsilon}\right|_{\epsilon=0}\left(\phi+\epsilon\frac{\left|x\right|^{2}}{2}\right)^{\ast}(x)\right]\\
 & = & -\phi^{\ast}(x).
\end{eqnarray*}

Therefore we get
\begin{eqnarray*}
w_{\alpha}(f) & = & \left.\frac{d}{d\epsilon}\right|_{\epsilon=0}\int\left(1+\frac{H(x,\epsilon)}{\beta}\right)^{-\beta}dx=\int\left.\frac{d}{d\epsilon}\right|_{\epsilon=0}\left(1+\frac{H(x,\epsilon)}{\beta}\right)^{-\beta}dx\\
 & = & \int-\beta\left(1+\frac{H(x,0)}{\beta}\right)^{-\beta-1}\cdot\frac{1}{\beta}\cdot\left(-\phi^{\ast}\right)dx\\
 & = & \int\phi^{\ast}(x)\cdot\left(1+\frac{\left|x\right|^{2}}{2\beta}\right)^{-\beta-1}dx
\end{eqnarray*}
 which is exactly what we wanted.
\end{proof}
Our next goal is to prove an Urysohn type inequality for $w_{\alpha}(f)$:
\begin{thm}
\label{pro:urysohn}If $f\in\Ca$ for $\alpha\ge-\frac{1}{n}$ then
\[
w_{\alpha}(f)\ge\int G_{\alpha}\cdot\left[\frac{n}{2}+\frac{1}{\kappa}\left(\frac{\int f}{\int G_{\alpha}}\right)^{\kappa}-\frac{1}{\kappa}\right],
\]
 where $\kappa=\frac{\alpha}{1+n\alpha}$\end{thm}
\begin{proof}
We can write
\[
\int G_{\alpha}\star\left[\epsilon\cdot f\right]=\int\left[\left(1-\epsilon\right)\cdot\left(\frac{1}{1-\epsilon}\cdot G_{\alpha}\right)\right]\star\left[\epsilon\cdot f\right],
\]
 and by Corollary \ref{cor:BBL-modified} we get
\[
\int G_{\alpha}\star\left[\epsilon\cdot f\right]\ge\left[\left(1-\epsilon\right)\left(\int\frac{1}{1-\epsilon}\cdot G_{\alpha}\right)^{\kappa}+\epsilon\left(\int f\right)^{\kappa}\right]^{\frac{1}{\kappa}},
\]
 and the first term in the right hand side can be calculated explicitly:
\begin{eqnarray*}
\int\frac{1}{1-\epsilon}\cdot G_{\alpha} & = & \int\left(1+\frac{\left|x\right|^{2}(1-\epsilon)}{2\beta}\right)^{-\beta}\\
 & = & n\omega_{n}\int_{0}^{\infty}r^{n-1}\left(1+\frac{r^{2}(1-\epsilon)}{2\beta}\right)^{-\beta}dr\\
 & = & \left(\frac{2\pi b}{1-\epsilon}\right)^{\frac{n}{2}}\frac{\Gamma\left(b-\frac{n}{2}\right)}{\Gamma\left(b\right)}=\left(1-\epsilon\right)^{-\frac{n}{2}}\int G_{\alpha}.
\end{eqnarray*}
Define 
\begin{eqnarray*}
A(\epsilon) & = & \int G_{\alpha}\star\left[\epsilon\cdot f\right]\\
B(\epsilon) & = & \left[\left(1-\epsilon\right)\left(\int\frac{1}{1-\epsilon}\cdot G_{\alpha}\right)^{\kappa}+\epsilon\left(\int f\right)^{\kappa}\right]^{\frac{1}{\kappa}}\\
 & = & \left[\left(1-\epsilon\right)^{1-\frac{\kappa n}{2}}\left(\int G_{\alpha}\right)^{\kappa}+\epsilon\left(\int f\right)^{\kappa}\right]^{\frac{1}{\kappa}}.
\end{eqnarray*}
We know that $A(0)=B(0)=\int G_{\alpha}$, and $A(\epsilon)\ge B(\epsilon)$
for every $\epsilon\ge0$. Hence we get
\[
w_{\alpha}(f)=A'(0)\ge B'(0),
\]
 and by direct computation
\begin{eqnarray*}
B'(0) & = & \frac{1}{\kappa}\left[\left(\int G_{\alpha}\right)^{\kappa}\right]^{\frac{1}{\kappa}-1}\cdot\left[-\left(1-\frac{\kappa n}{2}\right)\left(\int G_{\alpha}\right)^{\kappa}+\left(\int f\right)^{\kappa}\right]\\
 & = & \int G_{\alpha}\cdot\left[\frac{n}{2}+\frac{1}{\kappa}\left(\frac{\int f}{\int G_{\alpha}}\right)^{\kappa}-\frac{1}{\kappa}\right]
\end{eqnarray*}
 so we get the result.
\end{proof}
Notice that in the log-concave case $\kappa\to0$, and Theorem \ref{pro:urysohn}
reduces to the inequality 
\[
w_{0}(f)\ge\left(2\pi\right)^{\frac{n}{2}}\left[\frac{n}{2}+\log\left(\frac{\int f}{\int G}\right)\right]
\]
 from \cite{rotem_mean_2012}. 
\begin{rem}
In \cite{colesanti_area_2011}, Colesanti and Fragal\`{a} deal with
expressions of the form 
\[
\lim_{\epsilon\to0^{+}}\frac{\int g\star\left[\epsilon\cdot f\right]-\int g}{\epsilon}
\]
 where $g$ and $f$ are arbitrary log-concave functions. Among other
things, they prove analogs of Theorems \ref{thm:width-repr} and \ref{pro:urysohn},
assuming $f$ and $g$ are log-concave functions satisfying several
technical assumptions. It is not hard to extend their work to our
settings, and obtain results for $f$ and $g$ which are merely $\alpha$-concave.
Since the added difficulties are mostly technical, we will not pursue
the matter any further in this paper.
\end{rem}
Finally, we will demonstrate how one can obtain Poincaré type inequalities
by differentiating Urysohn's inequality. As far as we know this result
never appeared in print, even for the log-concave case.
\begin{thm}
\label{thm:poincare}Fix $\beta>n$. For any smooth function $\psi:\R^{n}\to\R$
which is bounded from below we have 
\begin{eqnarray*}
\int\left|\nabla\psi(x)\right|^{2}\cdot\left(1+\frac{\left|x\right|^{2}}{2\beta}\right)^{-\beta-1}dx & \ge & \frac{\kappa-1}{\int G_{\alpha}}\left[\int\psi(x)\left(1+\frac{\left|x\right|^{2}}{2\beta}\right)^{-\beta-1}dx\right]^{2}\\
 &  & +\frac{\beta+1}{\beta}\cdot\int\psi^{2}(x)\left(1+\frac{\left|x\right|^{2}}{2\beta}\right)^{-\beta-2}dx.
\end{eqnarray*}
 (as usual, $\kappa=\frac{\alpha}{1+n\alpha}=\frac{1}{n-\beta}$).\end{thm}
\begin{proof}
For $t\ge0$ define $\phi_{t}(x)=\frac{\left|x\right|^{2}}{2}+t\cdot\psi(x)$.
Since $\psi$ is bounded from below we know that $\phi_{t}>-\beta$
for small enough $t$. Hence we can define 
\[
f_{t}=\left(1+\frac{\phi_{t}}{\beta}\right)^{-\beta},
\]
 and 
\begin{eqnarray*}
A(t) & = & w_{\alpha}(f_{t})=\int\phi_{t}^{\ast}(x)\cdot\left(1+\frac{\left|x\right|^{2}}{2\beta}\right)^{-\beta-1}dx\\
B(t) & = & \int G_{\alpha}\cdot\left[\frac{n}{2}+\frac{1}{\kappa}\left(\frac{\int f_{t}}{\int G_{\alpha}}\right)^{\kappa}-\frac{1}{\kappa}\right].
\end{eqnarray*}
 We claim that $A(t)\ge B(t)$ for every (small enough) $t\ge0$.
Indeed, if $\phi_{t}$ happens to be convex, $f_{t}$ is $\alpha$-concave
and the claim follows from Theorem \ref{pro:urysohn}. In the general
case, replace $\phi_{t}$ by its convex envelope and notice that $B(t)$
increases, while $A(t)$ stays the same.

By inspecting the proof of \ref{pro:urysohn} or by direct computation,
we see that $A(0)=B(0)$. Let us calculate $A'(0),B'(0)$. 

For $A$, we will use the first variation formula for the Legendre
transform
\[
\dot{\phi}_{t}^{\ast}\left(\nabla\phi_{t}(x)\right)=-\dot{\phi}_{t}(x),
\]
 and by plugging our $\phi_{t}$ and $t=0$ we see that
\[
\dot{\phi}_{0}^{\ast}\left(x\right)=-\dot{\phi}_{0}(x)=-\psi(x),
\]
 so 
\[
A'(0)=-\int\psi(x)\left(1+\frac{\left|x\right|^{2}}{2\beta}\right)^{-\beta-1}dx.
\]
 For $B$ it is easy to differentiate directly: 
\begin{eqnarray*}
B'(t) & = & \int G_{\alpha}\cdot\frac{1}{\kappa}\cdot\kappa\cdot\left(\frac{\int f_{t}}{\int G_{\alpha}}\right)^{\kappa-1}\cdot\frac{1}{\int G_{\alpha}}\cdot\int\dot{f}_{t}\\
 & = & \left(\frac{\int f_{t}}{\int G_{\alpha}}\right)^{\kappa-1}\cdot\int\left[\left(-\beta\right)\left(1+\frac{\phi_{t}}{\beta}\right)^{-\beta-1}\cdot\frac{1}{\beta}\cdot\psi\right]\\
 & =- & \left(\frac{\int f_{t}}{\int G_{\alpha}}\right)^{\kappa-1}\int\psi(x)\left(1+\frac{\phi_{t}(x)}{\beta}\right)^{-\beta-1}dx,
\end{eqnarray*}
 and then for $t=0$ we get
\[
B'(0)=-\int\psi(x)\left(1+\frac{\left|x\right|^{2}}{2\beta}\right)^{-\beta-1}dx.
\]
 Therefore we have $A'(0)=B'(0)$, as was expected. 

It now follows that $A''(0)\ge B''(0)$. In order to calculate $A''(0)$
we will use the second variation formula
\[
\ddot{\phi}_{t}^{\ast}\left(\nabla\phi_{t}(x)\right)+\ddot{\phi}_{t}(x)=\left\langle \left(\text{Hess}\phi_{t}\right)^{-1}\nabla\dot{\phi}_{t}(x),\nabla\dot{\phi}_{t}(x)\right\rangle 
\]
 (for a proof of this formula, see for example \cite{cordero-erausquin_interpolations_2011}).
Plugging in $t=0$ we get
\[
\ddot{\phi}_{0}^{\ast}\left(x\right)+0=\left\langle \text{Id}^{-1}\cdot\nabla\psi(x),\nabla\psi(x)\right\rangle =\left|\nabla\psi(x)\right|^{2},
\]
 and then
\[
A''(0)=\int\left|\nabla\psi(x)\right|^{2}\cdot\left(1+\frac{\left|x\right|^{2}}{2\beta}\right)^{-\beta-1}dx.
\]

For $B$, we again have to differentiate directly and get
\begin{eqnarray*}
B''(t) & = & -(\kappa-1)\left(\frac{\int f_{t}}{\int G_{\alpha}}\right)^{\kappa-2}\cdot\frac{1}{\int G_{\alpha}}\cdot(-1)\cdot\left[\int\psi(x)\left(1+\frac{\phi_{t}(x)}{\beta}\right)^{-\beta-1}dx\right]^{2}\\
 &  & -\left(\frac{\int f_{t}}{\int G_{\alpha}}\right)^{\kappa-1}\int\psi(x)\left(-\beta-1\right)\left(1+\frac{\phi_{t}(x)}{\beta}\right)^{-\beta-2}\frac{1}{\beta}\cdot\psi(x)dx,
\end{eqnarray*}
 or, if we substitute $t=0$, we get
\[
B''(0)=\frac{\kappa-1}{\int G_{\alpha}}\left[\int\psi(x)\left(1+\frac{\left|x\right|^{2}}{2\beta}\right)^{-\beta-1}dx\right]^{2}+\frac{\beta+1}{\beta}\cdot\int\psi^{2}(x)\left(1+\frac{\left|x\right|^{2}}{2\beta}\right)^{-\beta-2}dx,
\]
 and the equality we wanted now follows.\end{proof}
\begin{cor}
For any smooth function $\psi:\R^{n}\to\R$ we have 
\[
\int\left|\nabla\psi(x)\right|^{2}d\gamma_{n}(x)\ge\int\psi^{2}(x)d\gamma_{n}(x)-\left[\int\psi(x)d\gamma_{n}(x)\right]^{2},
\]

where $d\gamma_{n}$ is the standard Gaussian probability measure
on $\R^{n}$.\end{cor}
\begin{proof}
This is simply the case $\beta=\infty$ of Theorem \ref{thm:poincare}.
By inspecting the proof of Theorem \ref{thm:poincare} we see that
in the case $\beta=\infty$ we do not need $\psi$ to be bounded from
below.

When $\beta\to\infty$ we have $\kappa\to0$, $\frac{\beta+1}{\beta}\to1$,
and
\[
\left(1+\frac{\left|x\right|^{2}}{2\beta}\right)^{-\beta-1},\left(1+\frac{\left|x\right|^{2}}{2\beta}\right)^{-\beta-2}\to e^{-\frac{\left|x\right|^{2}}{2}}=G(x).
\]
Thus we get 
\[
\int\left|\nabla\psi(x)\right|^{2}G(x)dx\ge\int\psi^{2}(x)G(x)dx-\frac{1}{\int G}\cdot\left[\int\psi(x)G(x)dx\right]^{2},
\]

and if divide both sides by $\int G$ we get exactly what we wanted.
\end{proof}
We see that Theorem \ref{thm:poincare} implies the Gaussian Poincaré
inequality, with a sharp constant. Hence, the case of general $\beta$
may be considered as a ``generalized Poincaré inequality''. At the
moment we are not aware of any applications for this generalized form.

\subsubsection*{Acknowledgment}

I would like to like to express my gratitude to Alexander Segal and
Boaz Slomka for providing some crucial insights for the proof of Theorem
\ref{thm:char-lc}. I would also like to thank my advisor, Prof. Vitali
Milman, for his help and support.

\bibliographystyle{plain}
\bibliography{a-concavity}

\begin{thebibliography}{10}

\bibitem{artstein-avidan_concept_2009}
S.~Artstein-Avidan and V.~Milman.
\newblock The concept of duality in convex analysis, and the characterization
  of the legendre transform.
\newblock {\em Annals of Mathematics. Second Series},
  169(2):661{\textendash}674, 2009.

\bibitem{artstein-avidan_characterization_2010}
S.~Artstein-Avidan and V.~Milman.
\newblock A characterization of the support map.
\newblock {\em Advances in Mathematics}, 223(1):379{\textendash}391, 2010.

\bibitem{avriel_r-convex_1972}
M.~Avriel.
\newblock r-convex functions.
\newblock {\em Mathematical Programming}, 2:309{\textendash}323, 1972.

\bibitem{bobkov_convex_2010}
S.~G. Bobkov.
\newblock Convex bodies and norms associated to convex measures.
\newblock {\em Probability Theory and Related Fields},
  147(1-2):303{\textendash}332, 2010.

\bibitem{borell_convex_1975}
C.~Borell.
\newblock Convex set functions in d-space.
\newblock {\em Periodica Mathematica Hungarica. Journal of the J\'{a}nos Bolyai
  Mathematical Society}, 6(2):111{\textendash}136, 1975.

\bibitem{brascamp_extensions_1976}
H.~J. Brascamp and E.~H. Lieb.
\newblock On extensions of the {{B}runn-{M}inkowski} and
  {{P}r\'{e}kopa-{L}eindler} theorems, including inequalities for log concave
  functions, and with an application to the diffusion equation.
\newblock {\em J. Functional Analysis}, 22(4):366{\textendash}389, 1976.

\bibitem{colesanti_area_2011}
A.~Colesanti and I.~Fragala'.
\newblock The area measure of log-concave functions and related inequalities.
\newblock {\em {arXiv:1112.2555}}, December 2011.

\bibitem{cordero-erausquin_interpolations_2011}
D.~Cordero-Erausquin and B.~Klartag.
\newblock Interpolations, convexity and geometric inequalities.
\newblock {\em {arXiv:1109.3652}}, September 2011.

\bibitem{gruber_endomorphisms_1991}
P.~M. Gruber.
\newblock The endomorphisms of the lattice of convex bodies.
\newblock {\em Abhandlungen aus dem Mathematischen Seminar der Universit\"{a}t
  Hamburg}, 61:121{\textendash}130, 1991.

\bibitem{klartag_geometry_2005}
B.~Klartag and V.~D. Milman.
\newblock Geometry of log-concave functions and measures.
\newblock {\em Geometriae Dedicata}, 112:169{\textendash}182, 2005.

\bibitem{rotem_mean_2012}
L.~Rotem.
\newblock On the mean width of log-concave functions.
\newblock In {\em Geometric aspects of functional analysis}, volume 2050 of
  {\em Lecture Notes in Math.} Springer, Berlin, 2012.

\bibitem{schneider_convex_1993}
R.~Schneider.
\newblock {\em Convex bodies: the Brunn-Minkowski theory}, volume~44 of {\em
  Encyclopedia of Mathematics and its Applications}.
\newblock Cambridge University Press, Cambridge, 1993.

\end{thebibliography}

\end{document}